\documentclass[12pt,twoside,a4paper]{article}

\usepackage{amsmath,amssymb,amsfonts,amsthm}
\usepackage{wasysym}
\usepackage{psfrag}
\usepackage{graphics,epsfig,amsmath}  
\usepackage{comment}
\usepackage{enumerate}
\usepackage{color}
\usepackage{appendix}
\usepackage{mathrsfs,bbold,dsfont}

\usepackage{setspace}

\usepackage{txfonts}

\theoremstyle{plain}
\newtheorem{theorem}{Theorem}
\newtheorem{lemma}[theorem]{Lemma}

\newtheorem{corollary}[theorem]{Corollary}

\theoremstyle{definition}

\newtheorem{definition}[theorem]{Definition}

\begin{document}

\title{Discounted Gottschalk-Hedlund theorem}

\author{Xifeng Su \footnote{School of Mathematical Sciences,
Beijing Normal University,
No. 19, XinJieKouWai St., HaiDian District,
 Beijing 100875, P. R. China,
\texttt{xfsu@bnu.edu.cn}}
\and
Philippe Thieullen \footnote{Institut de Math\'ematiques de Bordeaux,
Universit\'e de Bordeaux,
351, cours de la Lib\'eration - F 33405 Talence, France,
\texttt{philippe.thieullen@math.u-bordeaux1.fr}
}
} 

\date{\today}
\maketitle

\begin{abstract}
Ergodic optimization and discrete weak KAM theory are two parallel theories with several results in common. For instance, the Mather set is the locus of orbits which minimize the ergodic averages of a given observable. In the favorable cases, the observable is cohomologous to its ergodic minimizing value on the Mather set, and the discrete weak KAM solution plays the role of the transfer function. One possibility of construction of such a coboundary is by using the non linear Lax-Oleinik operator. The other possibility is by using a discounted cohomological equation. It is known that the discounted discrete weak KAM solution converges to some selected  weak KAM solution. We show that, in the ergodic optimization case for a coboundary observable over a minimal system, the discounted transfer function converges if and only if the observable is balanced.
\end{abstract}

\section{Notations and main statements}

We consider a {\it topological dynamical system} $(\Omega,\sigma)$ where $\Omega$ is a compact metric space and $\sigma:\Omega \to \Omega$ is a continuous map. We denote by $\mathscr{P}(\Omega,\sigma)$ the set of probability $\sigma$-invariant measures, and for every continuous function $f\in C^0(\Omega)$, by $\bar f$, the {\it ergodic minimizing value} of $f$
\begin{gather} \label{equation:ErgodicMinimizingValue}
\bar f := \min_{\mu \in \mathscr{P}(\Omega,\sigma)} \int\! f \, d\mu.
\end{gather}
A {\it minimizing measure} is a probability invariant measure realizing the minimum in \eqref{equation:ErgodicMinimizingValue}. We denote by $\mathscr{P}_{min}(\Omega,\sigma,f)$ the set of minimizing measures.

Given a continuous function $f:X \to \mathbb{R}$, we want to solve the following {\it cohomological equation} where $(M,u)$ are the two unknowns,
\begin{gather} \label{equation:CohomologicalEquation}
\begin{cases}
M \ \text{is a Borel invariant set,  and $\mu(M)=1$ for some $\mu \in \mathscr{P}(\Omega,\sigma)$}, \\
u:\Omega \to \mathbb{R} \ \text{is a non-negative Borel function}, \\
\forall \omega\in \Omega, \ f(\omega) - \bar f \geq u\circ(\omega) - u(\omega), \\
\forall \omega \in M, \ f(\omega)-\bar f = u \circ \sigma(\omega) - u(\omega).
\end{cases} \tag{CE}
\end{gather}
A function of the form $u \circ\sigma - u$ is called a {\it coboundary}, and $u$ is usually called a {\it transfer function}. 

Notice that such an invariant measure $\mu$ giving a unit mass to $M$  is necessarily a minimizing measure and satisfies $\text{\rm supp}(\mu) \subseteq \bar M$. As we are interested in the ``largest'' set $M$ for which such a transfer function $u$ exists, it is hence natural to consider the following set, called {\it Mather set} and  defined by
\begin{gather} \label{equation:MatherSet}
\mathcal{M}(f) := \bigcup \Big\{ \text{supp}(\mu) : \mu \in \mathscr{P}_{min}(\Omega,\sigma,f) \Big\}.
\end{gather}
It is easy to see that the Mather set is closed, invariant, and is equal to the support of some minimizing measure. The terminology ``Mather set'',  following Mather \cite{Mather1991} (where it is denoted $\text{\rm supp}\mathcal{M}_c$ before proposition 3), comes from the weak KAM theory initiated by Ma\~n\'e \cite{Mane1996} (Theorem B, cohomological equation on each $\text{\rm supp}(\mu)$), then extended  by Fathi \cite{Fathi1997} (theorem 1, sub-cohomological equation on the whole set $\Omega$) and later thoroughly studied by Fathi in \cite{Fathi2008} (the final terminology in section 4.12). 

For strongly regular systems, if the dynamical system $(\Omega,\sigma)$ is a Smale space \cite{Putnam2014} (for example a sub-shift of finite type)  and the function $f$ is Walters  \cite{Walters1978} (for example H\"older), then the cohomological equation (CE) admits a solution $(M,u)$ where $M=\mathcal{M}(f)$ and $u$ is Walters, see Bousch \cite{Bousch2001}. In an opposite direction, if $(\Omega,\sigma)$ is a topological dynamical system admitting invariant measures with different supports, for $C^0$ generic function $f$, every minimizing measure $\mu$ has full support, $\text{\rm supp}(\mu)=\Omega$, and there is no solution $(M,u)$ of (CE) with a continuous $u$, see Bousch \cite{Bousch2001}. There also exists  $C^{\infty}$ lacunary functions on the torus $f : \mathbb{T} \to \mathbb{R}$ and Liouville numbers $\alpha$ such that on the minimal and uniquely ergodic dynamical system $(\mathbb{T},R_\alpha)$, ($R_\alpha$ denotes the rotation by $\alpha$), there is no solution $(M,u)$ of (CE) with a Borel $u$, see Katok-Robinson \cite{KatokRobinson2001} (remarks 1 after theorem 3.5) and Herman \cite{Herman2004}.

\medskip
Our first result is the following.

\begin{theorem} \label{theorem:GottschalkHedlundMather}
Let $(\Omega,\sigma)$ be a topological dynamical system and $f : \Omega \to \mathbb{R}$ be a continuous function. Assume
\[
\forall \omega \in \Omega, \quad u(\omega) := - \inf_{n\geq1} \sum_{k=0}^{n-1} \big( f-\bar f \big)\circ \sigma^k(\omega) < +\infty.
\]
Let $u^+ := \max(u,0)$. Define a Borel set
\[
M := \Big\{ \omega \in \mathcal{M}(f) : \forall k\geq0, \ 
\big( f-\bar f - u^+\circ \sigma + u^+ \big) \circ\sigma^k(\omega) = 0
\Big\}.
\]
Then $(M,u^+)$ is a solution of the cohomological equation \eqref{equation:CohomologicalEquation}: 
\begin{enumerate}
\item \label{item:GottschalkHedlundMather-1} $u^+$ is lower semi-continuous,
\item \label{item:GottschalkHedlundMather-11} $\forall \omega\in \Omega, f(\omega)-\bar f \geq  u^+\circ \sigma(\omega) -  u^+ (\omega) $,
\item \label{item:GottschalkHedlundMather-2} $\forall \mu \in \mathscr{P}_{min}(\Omega,\sigma,f), \ \mu(M)=1$,
\item \label{item:GottschalkHedlundMather-3} $M$ is residual in $\mathcal{M}(f)$.
\end{enumerate}
\end{theorem}

The following corollary is an extension of Gottschalk-Hedlund theorem \cite{GottschalkHedlund1955} for every minimal subsets of the Mather set.

\begin{corollary} \label{corollary:GottschalkHedlundMather}
Let $(\Omega,\sigma)$ be a topological dynamical system, and $f : \Omega \to \mathbb{R}$ be a continuous function. Assume
\[
\exists C\geq0, \ \forall \omega \in\Omega, \ \forall n\geq1, \quad \sum_{k=0}^{n-1} ( f-\bar f )\circ \sigma^k(\omega) \geq -C.
\]
Then
\begin{enumerate}
\item \label{item:GottschalkHedlundMather-21} $\forall \omega \in \mathcal{M}(f), \ \forall n\geq1, \ \sum_{k=0}^{n-1} (f-\bar f) \circ \sigma^k(\omega) \leq C$,
\item \label{item:GottschalkHedlundMather-22} if $\mu$ is invariant and $\text{\rm supp}(\mu) \subset \mathcal{M}(f)$ then $\mu$ is minimizing,
\item \label{item:GottschalkHedlundMather-23}  there exists a lower semi-continuous function  $u:\Omega \to \mathbb{R}$ such that
\begin{enumerate}
\item $0\leq u \leq C$,
\item $\forall \omega \in \Omega,\ f(\omega) -\bar f \geq u\circ\sigma(\omega) -u(\omega)$,
\item for every minimal subset $X \subseteq \mathcal{M}(f)$, $u$ is continuous on $X$ and
\[
\forall \omega \in X, \quad  f(\omega) -\bar f = u\circ\sigma(\omega) -u(\omega).
\]
\end{enumerate}
\end{enumerate}
\end{corollary}

If $(\Omega,\sigma)$ is minimal, the Mather set must be equal to $\Omega$ and we recover the classical Gottshalk-Hedlund theorem.  The following statement is a slightly improved extension.

\begin{theorem}[Gottschalk-Hedlund \cite{GottschalkHedlund1955}] \label{theorem:GottschalkHedlund}
If $(\Omega,\sigma)$ is minimal, $f\in C^0(\Omega)$, and 
\[
\forall \omega \in\Omega, \quad \inf_{n\geq1} \sum_{k=0}^{n-1} ( f-\bar f )\circ \sigma^k(\omega) > -\infty,
\]
then $\exists u \in C^0(\Omega), \ \forall \omega \in \Omega, \ f(\omega) - \bar f =  u\circ \sigma(\omega) - u(\omega)$.
\end{theorem}

Notice that if $(\Omega,\sigma)$ is uniquely ergodic, $\mathcal{M}(f)=  \text{\rm supp}(\mu)$  and $\bar f = \int\! f\,d\mu$  for a unique ergodic measure $\mu$.

\medskip
We now consider a weaker form of the cohomological equation that we call {\it discounted cohomological equation}:
\begin{gather} \label{equation:DiscountedCohomologicalEquation}
\begin{cases}
\forall \epsilon>0, \ u_\epsilon : \Omega \to \mathbb{R} \ \text{is a $C^0$ function}, \\
\forall \epsilon>0, \ \forall \omega \in \Omega, \ f(\omega) = (1-\epsilon)u_\epsilon \circ \sigma (\omega) - u_\epsilon(\omega).
\end{cases} \tag{DCE}
\end{gather}
Notice that \eqref{equation:DiscountedCohomologicalEquation} has a unique solution, called {\it discounted transfer function}, and given by
\begin{gather} \label{equation:DiscountedSolution}
U_\epsilon[f](\omega) := -\sum_{k\geq0} (1-\epsilon)^k f\circ \sigma^k(\omega).
\end{gather}
We question whether the discounted solution $U_\epsilon[f]$ converges to some solution of (CE) as $\epsilon\to0$. We give a complete answer when $f$ is a  coboundary over a minimal system.

\begin{definition}
Let $(\Omega,\sigma)$ be a topological dynamical system, and $f:\Omega \to \mathbb{R}$ be a continuous function.
\begin{enumerate}
\item We say that $f$ is a {\it regular coboundary} if there exists a continuous function $u:\Omega \to \mathbb{R}$ such that $f = u\circ\sigma-u$.
\item We say that $f$ is a {\it balanced coboundary} if there exists a continuous function $u:\Omega \to \mathbb{R}$ such that $f = u\circ\sigma-u$ and $\int\! u \, d\mu$ is independent of $\mu \in \mathscr{P}(\Omega,\sigma)$.
\end{enumerate}
\end{definition}

Our second result is the following.

\begin{theorem} \label{theorem:DiscountedGottschalkHedlund}
Let $(\Omega,\sigma)$ be a topological dynamical system, and $f:\Omega \to \mathbb{R}$ be a regular coboundary.
\begin{enumerate}
\item \label{item:DiscountedGottschalkHedlund-1} If $f$ is balanced, then there exists a unique  $u\in C^0(\Omega)$ such that $f= u\circ\sigma-u$ and $\int\! u \,d\mu = 0, \ \forall \mu \in \mathscr{P}(\Omega,\sigma)$, and  $U_\epsilon[f] \to u$ uniformly in $\Omega$.
\item \label{item:DiscountedGottschalkHedlund-2} If  $(\Omega,\sigma)$ is minimal and $f$ is not balanced, then there exist $u\in C^0(\Omega)$ satisfying $f = u \circ \sigma - u$,  two ergodic invariant measures $\mu_0,\mu_1$  satisfying $\int\! u\, d\mu_0 \not= \int\! u\,d\mu_1$, and a residual set $M\subseteq \Omega$ such that, for every $\omega \in M$, there exists a decreasing sequence $(\epsilon_n)_{n\geq0}$ converging to 0  such that
\[
U_{\epsilon_{2p}}[f](\omega) \to u - \int\! u\,d\mu_0, \ U_{\epsilon_{2p+1}}[f](\omega) \to u - \int\! u\,d\mu_1.
\]
\end{enumerate}
\end{theorem}

The notion of discounted cohomological equation is reminiscent of the notion of discounted weak KAM solution discussed in \cite{DaviniFathiIturriagaZavidovique2016-b} in the continuous setting and in \cite{DaviniFathiIturriagaZavidovique2016-a,SuThieullen2018} in the discrete setting. Contrary to the phenomenon observed in theorem~\ref{theorem:DiscountedGottschalkHedlund}, the discounted weak KAM solution converges to some selected weak KAM solution, called balanced weak KAM solution, see \cite{SuThieullen2018} proposition 18 in the discrete setting.

\section{Proofs for the cohomological equation}

\begin{proof}[Proof of theorem \ref{theorem:GottschalkHedlundMather}]
Item \eqref{item:GottschalkHedlundMather-1} is a consequence of the fact that the supremum of continuous functions is lower semi-continuous.  

Item  \eqref{item:GottschalkHedlundMather-11} is an immediate consequence  of the following identity:
\begin{equation}\label{coboundary}
\forall \omega \in \Omega, \quad f(\omega) -\bar f = u^+\circ\sigma(\omega) - u(\omega).
\end{equation}

Indeed let $\omega \in \Omega$. Either 
\begin{gather*}
\begin{split}
(f-\bar f)(\omega) = -u(\omega) &< \sum_{k=0}^{n-1} \big(f-\bar f\big)\circ \sigma^k(\omega), \ \forall n\geq2, \\
&\leq \big(f-\bar f \big)(\omega) - u\circ\sigma(\omega),
\end{split} \\
u\circ\sigma(\omega) \leq0, \quad \big(f-\bar f\big)(\omega) = u^+\circ\sigma(\omega) - u(\omega),
\end{gather*}
or 
\begin{gather*}
\begin{split}
(f-\bar f)(\omega) \geq -u(\omega) &= \inf_{n\geq2} \sum_{k=0}^{n-1} \big(f-\bar f\big)\circ \sigma^k(\omega), \\
&=\big(f-\bar f \big)(\omega) - u\circ\sigma(\omega),
\end{split} \\
u\circ\sigma(\omega) \geq0, \  \ \big(f-\bar f\big)(\omega) = u^+\circ\sigma(\omega) - u(\omega).
\end{gather*}
We have proved in particular,
\[
 \forall \omega \in \Omega, \quad \big(f-\bar f\big)(\omega) \geq u^+\circ\sigma(\omega)-u^+(\omega).
\] 
Notice that it implies $u^+\circ\sigma - u^+\in L^1(\mu)$ and $\int (u^+\circ\sigma-u^+) \,d\mu=0, \ \forall \mu \in \mathscr{P}(\Omega,\sigma)$. 

The proof of item \eqref{item:GottschalkHedlundMather-2} will follow from the fact that $u\geq0, \ \mu(d\omega) \ \text{a.e.}$ for every $\mu \in \mathscr{P}_{min}(\Omega,\sigma,f)$. Let $u^-:= (-u)^+$ and $\mu$ be a minimizing measure. We have
\begin{gather*}
0 = \int\!\big(f-\bar f\big) \, d\mu = \int\! (u^+\circ\sigma-u) \,d\mu = \int\! (u^+\circ\sigma -u^+) + u^- \,d\mu, \\
\int\! u^- \,d\mu=0 \ \ \Rightarrow \ \  u\circ\sigma^k(\omega)\geq0, \ \mu(d\omega), \ \forall k\geq0, \ \  \text{a.e.}
\end{gather*}which implies $\mu(M)=1$.

The proof of item \eqref{item:GottschalkHedlundMather-3} will follow from the fact that $u\geq0$ on the set $R$ of points of continuity of $u$ belonging to  the Mather set, and that $R$ is residual in the Mather set thanks to the lower semi-continuity of $u$. Indeed let $\omega \in R$. Then $\omega \in \text{\rm supp}(\mu)$ for some minimizing measure $\mu$. By contradiction, if $u(\omega)<0$, we would have $u<0$ on a neighborhood $U$ containing $\omega$. Since $U \cap \text{\rm supp}(\mu) \not=\emptyset$, we would have $\mu(U)>0$, contradicting $u\geq0, \ \mu$ a.e. Therefore, $u\geq0$ for any $\omega\in R$, which implies \eqref{coboundary} holds with $u^+(\omega)$ instead of $u(\omega)$. Hence, the residual set $\cap_{k\geq0} \sigma^{-k}(R)$ is contained in $M$, which completes the proof of \eqref{item:GottschalkHedlundMather-3}.
\end{proof}

\begin{proof}[Proof of corollary \ref{corollary:GottschalkHedlundMather}]
Theorem \ref{theorem:GottschalkHedlundMather} implies the existence of a lower semi-continuous function $u:\Omega  \to \mathbb{R}$ and a residual subset $M \subseteq \mathcal{M}(f)$ such that
\begin{itemize}
\item $\forall \omega \in\Omega, \ f(\omega) - \bar f \geq u\circ\sigma(\omega) - u(\omega)$,
\item $\forall \omega \in M, \  f(\omega) - \bar f = u\circ\sigma(\omega) - u(\omega)$,
\item $\forall \omega \in \Omega, \ 0\leq u(\omega) \leq C$.
\end{itemize}

The proof of item \eqref{item:GottschalkHedlundMather-21} follows from,
\[
\forall \omega \in M, \ \forall n\geq1, \ \sum_{k=0}^{n-1} \big(f - \bar f \big) \circ \sigma^k(\omega) = u^+ \circ \sigma^n(\omega) - u^+(\omega) \leq u^+ \circ \sigma^n(\omega) \leq C,
\]
and from the fact that $M$ is residual and in particular dense in the Mather set.

The proof of item  \eqref{item:GottschalkHedlundMather-22} follows from item   \eqref{item:GottschalkHedlundMather-21}. If $\text{\rm supp}(\mu) \subseteq \mathcal{M}(f)$, then
\[
\forall n\geq0, \ n\int\! f \,d \mu \leq n \bar f+C \ \Rightarrow \ \int\!f\,d\mu = \bar f.
\]

The proof of item  \eqref{item:GottschalkHedlundMather-23} follows from  theorem \ref{theorem:GottschalkHedlundMather} applied to $-f$ on any $(X,\sigma)$. Indeed, thanks to item \eqref{item:GottschalkHedlundMather-21}, we have
\begin{gather*}
\sup_{\mu \in \mathscr{P}(X,\sigma)} \int\! f \,d\mu = \bar f \ \ \text{and} \ \ 
\forall \omega \in X, \quad  \sup_{n\geq1} \sum_{k=0}^{n-1} \big( f-\bar f \big)\circ \sigma^k(\omega) < +\infty.
\end{gather*}
There exists a non-positive  upper semi-continuous function $v: X \to \mathbb{R}$ such that
\begin{gather*}
\forall \omega \in X, \ f(\omega)-\bar f \leq v \circ\sigma(\omega)-v(\omega).
\end{gather*}
Then
\[
\forall \omega \in X, \quad u \circ \sigma(\omega) - u(\omega) \leq f\omega) - \bar f \leq  v \circ \sigma(\omega) - v(\omega).
\]
Since $u-v$ is lower semi-continuous on $X$, $u-v$ attains its infimum on $X$. Define 
\[
D := \min \big\{ (u-v)(\omega) : \omega \in X \big\}, \ \  \tilde{X} := \{\omega \in X : (u-v)(\omega)\leq D\}.
\]
Since $(u-v) \circ \sigma \leq (u-v)$, $\tilde{X}$ is compact, $\sigma$-invariant, therefore by minimality is equal to $X$: $u-v=D$ on $X$, $u$ and $v$ restricted to the $X$ are continuous and $f - \bar f = u \circ \sigma- u = v \circ \sigma - v$ on $X$.
\end{proof}

We will need the following lemma for the proof of theorem \ref{theorem:GottschalkHedlund}. See proposition A.7 in Morris \cite{Morris2013} for a proof.

\begin{lemma} \label{lemma:Morris}
Let $(\Omega,\sigma)$ be a topological dynamical system and $f\in C^0(\Omega)$. Then
\[
\exists \omega_* \in \Omega, \ \forall n\geq 0, \   \frac{1}{n} \sum_{k=0}^{n-1} f\circ \sigma^k(\omega_*) \leq  \min_{\mu \in\mathscr{P}(\Omega,\sigma)} \int\! f\,d\mu.
\]
\end{lemma}

\begin{proof}[Proof of theorem \ref{theorem:GottschalkHedlund}]
It follows from lemma \ref{lemma:Morris} and by assumption of the theorem, there exists $\omega_* \in \Omega$ and a constant $C\geq0$ such that 
\[
\forall n\geq0, \quad  -C \leq \sum_{k=0}^{n-1} (f-\bar f)\circ\sigma^k(\omega_*) \leq 0.
\]
Then
\[
\forall m,n\geq0, \quad  \sum_{k=m}^{m+n-1} (f-\bar f)\circ\sigma^k(\omega_*) \geq -C.
\]
By minimality of $(\Omega,\sigma)$, the orbit of $\big(\sigma^k(\omega_*)\big)_{k\geq0}$ is dense,
\[
\forall \omega \in\Omega, \   \forall n\geq1, \ \sum_{k=0}^{n-1} ( f-\bar f )\circ \sigma^k(\omega) \geq -C.
\]
We conclude the proof by using corollary \ref{corollary:GottschalkHedlundMather}.
\end{proof}

\section{Proofs for the discounted cohomological equation}

Notice that the unique solution of \eqref{equation:DiscountedCohomologicalEquation}, equation \eqref{equation:DiscountedSolution}, can be written as
\begin{gather*}
U_\epsilon[f](\omega)  = -\frac{1}{\epsilon} \int\! f \, d\mu_{\epsilon,\omega}, \quad \forall \omega \in \Omega,
\end{gather*}
where $\mu_{\epsilon,\omega} := \sum_{k\geq0} \epsilon(1-\epsilon)^k \delta_{\sigma^k(\omega)}$ is a probability measure not necessarily invariant. 

\medskip
The proof of item \eqref{item:DiscountedGottschalkHedlund-1} of theorem \ref{theorem:DiscountedGottschalkHedlund} follows from the following lemma.

\begin{lemma} \label{lemma:DiscountedMeasure}
Let be $f \in C^0(\Omega)$.
\begin{enumerate}
\item \label{item:DiscoutedMeasure-1} If $\forall \mu \in \mathscr{P}(\Omega,\sigma), \ \int\! f \,d\mu = 0$, then $\int\! f \, d\mu_{\epsilon,\omega} \to 0$ uniformly in $\omega \in \Omega$.
\item \label{item:DiscoutedMeasure-2} If $f=u \circ \sigma - u$, then $U_\epsilon[f](\omega) =  u(\omega) - \int\! u \circ \sigma\,d\mu_{\epsilon,\omega}, \ \forall \omega \in \Omega$.
\end{enumerate}
\end{lemma}

\begin{proof}[ Proof of item \eqref{item:DiscoutedMeasure-1}]
We first prove that 
\[
\limsup_{\epsilon\to0}\sup_{\omega \in \Omega} \int\! f\,d\mu_{\epsilon,\omega} = 0.
\]
Let $(\epsilon_n)_{n\geq0}$ be a sequence tending to 0 and realizing the above $\limsup$. Let $(\omega_n)_{n\geq0}$ be a sequence of points of $\Omega$ realizing the supremum of $\int\!  f\,d\mu_{\epsilon_n,\omega}$  for each $\epsilon_n$. Choose a sub-sequence of $(\epsilon_n)_{n\geq0}$, that we denote in the same way, such that $(\mu_{\epsilon_n,\omega_n})_{n\geq0}$ converges to some probability measure $\mu$. Notice that
\[
\forall n \geq0, \ \forall g \in C^0(\Omega), \quad (1-\epsilon_n) \mu_{\epsilon_n,\omega_n}(g\circ\sigma) = \mu_{\epsilon_n,\omega_n}(g) -\epsilon_n g(\omega_n).
\]
Taking $n\to+\infty$, we obtain $\mu \in \mathscr{P}(\Omega,\sigma)$ and 
\[
\limsup_{\epsilon\to0}\sup_{\omega \in \Omega} \int\! f\,d\mu_{\epsilon,\omega} = \int\! f \, d\mu_{\epsilon_n,\omega_n} \to \int\! f \,d\mu=0.
\] 
Similarly we show $\liminf_{\epsilon\to0}\inf_{\omega \in \Omega} \int\! f\,d\mu_{\epsilon,\omega} = 0$. Item \eqref{item:DiscoutedMeasure-1} is proved.
\end{proof}

\begin{proof}[Proof of item \eqref{item:DiscoutedMeasure-2}]
We observe
\begin{align*}
u &= u\circ\sigma -f  = (1-\epsilon)u\circ\sigma-f + \epsilon u\circ\sigma \\
&= \sum_{k\geq0} (1-\epsilon)^k (-f+\epsilon u\circ\sigma) \circ \sigma^k \\
u(\omega) &= U_\epsilon[f](\omega) + \int\! u\circ\sigma \,d\mu_{\epsilon,\omega}. \qedhere
\end{align*}
\end{proof}

\begin{proof}[Proof of item \eqref{item:DiscountedGottschalkHedlund-1} of theorem \ref{theorem:DiscountedGottschalkHedlund}]
If $f$ is a balanced coboundary, $f = u\circ\sigma-u$ for some $u$ satisfying $\int\! u\,d\mu = 0, \ \forall \mu \in\mathcal{M}(\Omega,\sigma)$. Then, thanks to lemma \ref{lemma:DiscountedMeasure},
\begin{gather*}
U_\epsilon[f](\omega) = u(\omega) - \int\! u\circ\sigma \, d\mu_{\epsilon,\omega} \to u(\omega), \quad \text{uniformly in $\omega \in\Omega$}. 
\end{gather*}
In particular, such a transfer function $u$ is unique.
\end{proof}

The proof of the second item of theorem \ref{theorem:DiscountedGottschalkHedlund} will be given after the two following lemmas.

\begin{lemma} \label{lemma:ResidualSet}
Let $(\Omega,\sigma)$ be a minimal dynamical system, and $\mu_0,\mu_1$ be two ergodic measures. Then there exists a residual subset $M\subseteq \Omega$ such that for every $\omega \in M$ there exists a sequence of integers $(N_p)_{p\geq1}$ such that 
\begin{gather*}
\forall p \geq1, \quad N_{2p} < \ln(N_{2p+1}), \ N_{2p-1} < \ln(N_{2p}), \\
\forall p \geq 1, \ \forall \ln(N_p) < n < N_p, \quad \Big| \frac{1}{n}\sum_{k=0}^{n-1} u\circ\sigma^k(\omega) - \mu_{[p]}(u) \Big| < \frac{1}{p},
\end{gather*}
where $[p] = p \mod 2$.
\end{lemma}

\begin{proof}
Let $i=0,1$.  As $\mu_i$ is ergodic, thanks to Birkhoff's ergodic theorem, for every $p\geq1,q\geq1$, 
\[
U_{p,q}^{(i)} := \Big\{ \omega \in \text{\rm supp}(\mu_i) : \exists N\geq q, \ \forall \ln(N) < n <N, \ \Big|\frac{1}{n} \sum_{k=0}^{n-1}u\circ\sigma^{k+1}(\omega)-\mu_i(u) \Big| < \frac{1}{p}\Big\}
\]
is an open and dense set of $\text{\rm supp}(\mu_i)$. As $(\Omega,\sigma)$ is minimal, $\text{\rm supp}(\mu_i)=\Omega$. The set $M_i := \cap_{p,q\geq1} U_{p,q}^{(i)}$ is thus a residual set of $\Omega$. Define $M := M_0 \cap M_1$. Then $M$ is a residual set. If $\omega \in M$, we construct by induction a sequence of integers $(N_p)_{p\geq1}$ satisfying the properties of the above lemma:
\begin{gather*}
p=1,q=1, \ \exists N_1, \  \forall \ln(N_1) < n< N_1, \ \Big|\frac{1}{n} \sum_{k=0}^{n-1}u\circ\sigma^{k+1}(\omega)-\mu_1(u) \Big| < 1,\\
p=2,q=\ln(N_1), \ \exists N_2> q, \  \forall \ln(N_2) < n <  N_2, \ \Big|\frac{1}{n} \sum_{k=0}^{n-1}u\circ\sigma^{k+1}(\omega)-\mu_0(u) \Big| < \frac{1}{2}, \\
p=3,q=\ln(N_2), \ \exists N_3> q, \ \forall \ln(N_3) < n <  N_3, \ \Big|\frac{1}{n} \sum_{k=0}^{n-1}u\circ\sigma^{k+1}(\omega)-\mu_1(u) \Big| < \frac{1}{3},
\end{gather*}
and so on.
\end{proof}

Denote by $A_{n,\omega} := \frac{1}{n}\sum_{k=0}^{n-1} \delta_{\sigma^k(\omega)}$ the empirical measure.

\begin{lemma} \label{lemma:Decomposition}
For every $\epsilon>0, \ \omega \in \Omega,\ n\geq2$
\begin{gather*}
\mu_{\epsilon,\omega} = \sum_{k=\lfloor \ln(n) \rfloor}^{n-2}(k+1)\epsilon^2(1-\epsilon)^kA_{k+1,\omega} + R_{n,\epsilon,\omega}, \\
R_{n,\epsilon,\omega} = \sum_{k=0}^{\lfloor \ln(n) \rfloor-1} (k+1)\epsilon^2(1-\epsilon)^k A_{k+1,\omega} + n\epsilon(1-\epsilon)^{n-1} A_{n,\omega} + (1-\epsilon)^n\mu_{\epsilon,\sigma^n(\omega)}, \\
\sup_{\omega}R_{n,\epsilon,\omega}(\mathds{1}) \leq (\epsilon \ln(n))^2 + (1+ \epsilon n e )e^{-\epsilon n}.
\end{gather*}
\end{lemma}

\begin{proof}
We have
\begin{gather*}
\mu_{\epsilon,\omega} = \sum_{k=0}^{n-1}\epsilon(1-\epsilon)^k\delta_{\sigma^k(\omega)} + (1-\epsilon)^n\mu_{\epsilon,\sigma^n(\omega)}, \\
\begin{split}
\sum_{k=0}^{n-1}\epsilon(1-\epsilon)^k\delta_{\sigma^k(\omega)} &= \sum_{k=0}^{n-1}\epsilon(1-\epsilon)^k \big((k+1)A_{k+1,\omega} - kA_{k,\omega} \big) \\
&=\sum_{k=0}^{n-2} (k+1)\epsilon^2(1-\epsilon)^kA_{k+1,\omega} + n\epsilon(1-\epsilon)^{n-1}A_{n,\omega}. \qedhere
\end{split}
\end{gather*}
\end{proof}

\begin{proof}[Proof of item \eqref{item:DiscountedGottschalkHedlund-2} of theorem \ref{theorem:DiscountedGottschalkHedlund}]
Let $(\Omega,\sigma)$ be a minimal dynamical system and $f$ be a non-balanced coboundary: $f = u\circ\sigma - u$, $\int\! u\,d\mu_0 \not= \int\! u \,d\mu_1$ for some ergodic measures $\mu_0,\mu_1$. Let $M$ be the residual set given by lemma \ref{lemma:ResidualSet}. Let $\omega \in M$ and $(N_p)_{p\geq1}$ be the sequence of integers given by lemma \ref{lemma:ResidualSet}. Let $\epsilon_p := \frac{\ln(N_p)}{N_p}$. Define
\[
\alpha_p := \sum_{k=\lfloor \ln(N_p) \rfloor}^{N_p-2}(k+1)\epsilon_p^2(1-\epsilon_p)^k.
\]
Then, using lemma \ref{lemma:Decomposition},
\begin{gather*}
0 \leq 1 - \alpha_p = \sup_\omega R_{N_p,\epsilon_p,\omega} \leq  (\epsilon_p \ln(N_p))^2 + (1+ \epsilon_p N_p e )e^{-\epsilon_p N_p} \to 0, \\
\begin{split}
\mu_{\epsilon_p,\omega}(u\circ\sigma) - \mu_{[p]}(u) &=  \sum_{k=\lfloor \ln(N_p) \rfloor}^{N_p-2}(k+1)\epsilon^2(1-\epsilon)^k \big( A_{k+1,\omega}(u\circ\sigma) - \mu_{[p]}(u) \big) \\ 
&\quad\quad  + R_{N_p,\epsilon_p,\omega}(u\circ\sigma) - (1-\alpha_p) \mu_{[p]}(u), \\
| \mu_{\epsilon_p,\omega}(u\circ\sigma) - \mu_{[p]}(u)  | &\leq \frac{\alpha_p}{p} + 2 (1-\alpha_p) \|u\|_\infty \to 0.
\end{split}
\end{gather*}
We conclude the proof of the theorem using item \eqref{item:DiscoutedMeasure-2} of lemma \ref{lemma:DiscountedMeasure}.
\end{proof}


\end{document}